\documentclass[12pt, a4paper]{amsart}
\pdfoutput=1
\usepackage[T1]{fontenc}
\usepackage{amsmath,amssymb,amsthm}
\usepackage{url}
\usepackage{fancyvrb}
\usepackage{tikz}
\usetikzlibrary{automata}

\newtheorem{theorem}{Theorem}
\newtheorem{corollary}[theorem]{Corollary}
\theoremstyle{remark}
\newtheorem{remark}[theorem]{Remark}

\newcommand{\abs}[1]{\lvert #1\rvert}
\newcommand{\calA}{\mathcal{A}}
\newcommand{\E}{\mathbb{E}}
\renewcommand{\MR}[1]{}
\renewcommand{\P}{\mathbb{P}}
\newcommand{\V}{\mathbb{V}}

\newcommand{\TODO}[1]%
{\par\fbox{\begin{minipage}{0.9\linewidth}\textbf{TODO:} #1\end{minipage}}\par}

\title[Waiting Time Distributions of Runs]{Application of Smirnov Words to Waiting Time Distributions of Runs}
\author{Uta Freiberg}
\address{Institut f\"ur Stochastik und Anwendungen\\
Universit\"at Stuttgart\\
Pfaffenwaldring 57\\
D-70569 Stuttgart\\
Germany}
\email{uta.freiberg@mathematik.uni-stuttgart.de}
\thanks{Parts of the
  article were written while Uta Freiberg was a visitor at
  Stellenbosch University.}
\author{Clemens Heuberger}
\address{Institut f\"ur Mathematik\\
Alpen-Adria-Universit\"at Klagenfurt\\
Universit\"atsstra\ss e 65--67\\9020 Klagenfurt\\Austria}
\email{clemens.heuberger@aau.at}
\thanks{Clemens Heuberger is supported by the Austrian Science Fund (FWF):
  P~24644-N26.
  Parts of the
  article were written while Clemens Heuberger was a visitor at
  Stellenbosch University.}

\author{Helmut Prodinger}
\address{Department of Mathematical Sciences\\Stellenbosch University\\7602 Stellenbosch\\
 South Africa}
\email{hproding@sun.ac.za}
\thanks{Helmut Prodinger
  is supported by an incentive grant of the NRF of South Africa.}


\begin{document}
\begin{abstract}
  Consider infinite random words over a finite alphabet where the letters occur as an i.i.d.\ sequence according to some arbitrary distribution on the alphabet.
  The expectation and the variance of
  the waiting time for the first completed $h$-run of any letter (i.e., first
  occurrence of $h$ subsequential equal letters) is computed.

  The expected waiting time for the completion of $h$-runs of $j$ arbitrary
  distinct letters is also given.
\end{abstract}
\maketitle

\section{Introduction}
%%%%%%%%%%%%%%%%%
In \cite{Szekely:1986:parad}, the following paradox is presented: In measuring the regularity of a die one may use waiting times for sequences of the same side of certain lengths. For example, if ones throws a regular six-sided die, it takes $7$ throws on average to get a number subsequently twice and $43$ throws to get a number three times in succession. Heuristically, one would expect that a \textit{smaller} number of throws is needed to get such sequences with a \textit{biased} die. This leads to the definition to call one die  \textit{more regular} than another one if more throws are needed to get sequences of one side of a certain length. Now the paradox is that there exist dice---say $A$ and $B$---where the mean waiting time for two digits in a row is longer for die $A$ while the mean waiting time for three digits in a row is longer for die $B$ (an example has been given by M\'ori, see \cite[p.~62]{Szekely:1986:parad}). The consequence of this paradox is that one cannot use the mean waiting times for such runs as a (sufficient) criterion for the definition of regularity of a die (or whatever random sequence of digits from a finite alphabet).

This paradox gave motivation to calculate first and second moments of such waiting times for so called \textit{$h$-runs}. In particular, the formula for the first moment of the waiting time for the first completed $h$-run of \textit{any} digit---which was already given in \cite{Szekely:1986:parad}---is proved without using the strong law of large numbers or any other limit theorem (see Theorem~\ref{theorem:first-h-run-expectation}). Moreover, the variance of the waiting time for the first completed $h$-run is presented in the same theorem. We then compute the waiting time for the completion of $h$-runs of $j$ different letters in Theorem~\ref{theorem:j-h-runs-expectation-general}. In particular, for $j=r$ (the number of possible letters), we get results about the waiting time for a full collection of runs.

Our fundamental technique is the calculation of generating functions of such waiting times; our main trick is the combination of two very useful observations: Firstly, we make use of the very simple but crucial identity \eqref{eq:Y_n=B_j} (see~\cite{Flajolet-Gardy-Thimonier:1992:birth}) which already has been a powerful tool in the treatment of the coupon collector problem and/or the birthday paradox. Secondly, we use the generating function of Smirnov words (see \cite{Flajolet-Sedgewick:ta:analy}) to count words with a limited number of repetitions of single letters using an appropriate substitution.

We conclude the paper in Section~\ref{sec:algorithmic} with an algorithmic approach for specific situations.

%%%%%%%%%%%%%%%%%%
\section{Preliminaries}\label{section:preliminaries}
We consider infinite words $X_1X_2\ldots$ over the alphabet $\calA=\{1, \ldots, r\}$
where the random variables $X_i$ are i.i.d.\ with $\P\{X_i=k\}=p_k>0$ for some
$p_1$, \ldots, $p_r$.

We say that a letter $\ell\in\calA$ has an \emph{$h$-run in $X_1\ldots
X_n$} if there are $h$ consecutive letters $\ell$ in the word $X_1\ldots X_n$,
or in other words, if the word $\ell^h=\ell\ell\ldots\ell$ (with $h$
repetitions) is a factor of the word $X_1\ldots X_n$.

We consider the random variable $B_j$ giving the first position $n$ such that
there exist $j$ of the $r$ letters having an $h$-run in $X_1\ldots X_n$. This is a
random variable on the infinite product space consisting of all infinite words
endowed with the product measure.

On the other hand, we consider the random variable $Y_n$ counting the number of
letters which had an $h$-run in $X_1\ldots X_n$. This is a random variable on
the finite product space consisting of all words of length $n$, again with its
product measure.

By construction, we have
\begin{equation}\label{eq:Y_n=B_j}
\P\{Y_n\ge j\}=\P\{B_j\le n\},
\end{equation}
cf.~\cite[Eqn.~(6)]{Flajolet-Gardy-Thimonier:1992:birth}.
As a consequence, we obtain (cf.~\cite[Eqn.~(7)]{Flajolet-Gardy-Thimonier:1992:birth})
\begin{equation}\label{eq:expectation-general}
\E(B_j)=\sum_{n\ge 0}\P\{B_j>n\}=\sum_{n\ge 0}\P\{Y_n<j\}=\sum_{q=0}^{j-1}\sum_{n\ge0}\P\{Y_n= q\}.
\end{equation}
With the generating function
\begin{equation}\label{eq:definition-G_j-z}
  G_j(z) = \sum_{n\ge 0} \P\{Y_n<j\}z^n,
\end{equation}
this amounts to
\begin{equation*}
  \E(B_j) = G_j(1).
\end{equation*}

To compute the variance, we first note that
\begin{align*}
  \E(B_j^2) &= \sum_{n\ge 0} n^2 \P\{B_j = n \}
  = \sum_{n\ge 0} n^2 (\P\{B_j > n-1 \} - \P\{B_j > n \})\\
  &=\sum_{n\ge 0} (n+1)^2 \P\{B_j > n \} - \sum_{n\ge 0} n^2 \P\{B_j > n \}\\
  &= \sum_{n\ge 0}(2n+1) \P\{B_j > n \}
  = \sum_{n\ge 0}(2n+1) \P\{Y_n <j \}\\
  &=2 G_j'(1) + G_j(1)
\end{align*}
where we used~\eqref{eq:Y_n=B_j} and the definition of $G_j(z)$ given
in~\eqref{eq:definition-G_j-z}. We conclude that
\begin{equation}\label{eq:variance}
  \V(B_j)=\E(B_j^2)-\E(B_j)^2 = 2G_j'(1) + G_j(1) - G_j(1)^2.
\end{equation}

A \emph{Smirnov word} is defined to be any word which has no consecutive equal
letters. The ordinary generating function of Smirnov words over the alphabet
$\calA$ is
\begin{equation}\label{eq:Smirnov}
S(v_1,\dots,v_r)=\frac1{\displaystyle 1-\sum_{i=1}^r\frac{v_i}{1+v_i}}
\end{equation}
where $v_i$ counts the number of occurrences of the letter $i$, cf.\ Flajolet
and Sedgewick~\cite[Example~III.24]{Flajolet-Sedgewick:ta:analy}.

\section{Moments of the first \texorpdfstring{$h$}{h}-run}
In this section, we study the first occurrence of any $h$-run.
In the framework of Section~\ref{section:preliminaries}, this corresponds to
the case $j=1$ and the random variable $B_1$.

We prove the following result on the expectation of $B_1$:

\begin{theorem}\label{theorem:first-h-run-expectation}If $p_i<1$ for $1\le i\le r$, the expectation and the
  variance of the first occurrence of an $h$-run are
  \begin{align}
    \E(B_1)&
    =\frac1{\displaystyle\sum_{i=1}^r\frac1{p_i^{-1}+\cdots+p_i^{-h}}} \label{eq:first-h-run-expectation}\\
    \intertext{and}
    \V(B_1)&=\frac{\displaystyle\sum_{i=1}^r\biggl(\frac{p_i +p_i^h}{1-p_i^h} - 2h\frac{ p_i^h(1-p_i)}{(1-p_i^h)^2}\biggr)}{\displaystyle\biggl(\sum_{i=1}^r\frac1{p_i^{-1}+\cdots+p_i^{-h}}\biggr)^2}.\label{eq:first-h-run-variance}
  \end{align}
\end{theorem}
The result~\eqref{eq:first-h-run-expectation} on the expectation also appears (without proof) in \cite[p.~62]{Szekely:1986:parad}.
Each summand of the numerator of \eqref{eq:first-h-run-variance} is indeed non-negative, because this is equivalent to
\begin{equation*}
  \frac{p_i + p_i^h}2\cdot \frac{1+p_i+\cdots + p_i^{h-1}}{h}\ge p_i^h,
\end{equation*}
which is true by the inequality between the arithmetic and the geometric mean, applied to both factors.

\begin{proof}[Proof of Theorem~\ref{theorem:first-h-run-expectation}]
\allowdisplaybreaks
In the case $j=1$, \eqref{eq:expectation-general} reads
\begin{equation}\label{eq:expectation-first-run}
\E(B_1)=\sum_{n\ge0}\P\{Y_n= 0\}.
\end{equation}
Thus we have to determine the probability that a word of length $n$ does not
have any $h$-run. Such words arise from a Smirnov word by replacing single
letters by runs of length in $\{1, \ldots, h-1\}$ of the same letter.

In terms of generating function, this corresponds to replacing each $v_i$ by
\begin{equation*}
p_i z+\cdots+(p_i z)^{h-1}=\frac{p_i z-(p_i z)^h}{1-p_i z}.
\end{equation*}
Here, $z$ marks the length of the word. We obtain
\begin{align*}
  G_1(z)&=\sum_{n\ge 0}\P\{Y_n=0\}z^n=S\left(\frac{p_1 z-(p_1 z)^h}{1-p_1 z}, \ldots,
    \frac{p_r z-(p_r z)^h}{1-p_r z}\right)\\
  &=\frac1{\displaystyle
    1-\sum_{i=1}^r\frac{\frac{p_i z-(p_i z)^{h}}{1-p_i z}}{1+\frac
      {p_i z-(p_i z)^{h}}{1-p_i z}}}
  =\frac1{1-\displaystyle \sum_{i=1}^r\frac{p_i z-(p_i z)^{h}}{1-(p_i z)^{h}}}.
\end{align*}

By \eqref{eq:expectation-first-run}, we are only interested in $z=1$:
\begin{equation*}
\E(B_1)=\sum_{n\ge0}\P\{Y_n= 0\}=G_1(1)=\frac1{1- \sum_{i=1}^r\frac{p_i-p_i^{h}}{1-p_i^{h}}}.
\end{equation*}
Replacing the summand $1$ in the denominator by $p_1+\cdots+p_r$ yields
\begin{align*}
  \E(B_1) &= \frac1{\displaystyle\sum_{i=1}^r
    \Bigl(p_i-\frac{p_i-p_i^{h}}{1-p_i^{h}}\Bigr)}
  =\frac1{\displaystyle\sum_{i=1}^r\frac{p_i-p_i^{h+1}-p_i+p_i^{h}}{1-p_i^{h}}}\\
  &=\frac1{\displaystyle\sum_{i=1}^r\frac{p_i^h(1-p_i)}{1-p_i^{h}}}
  =\frac1{\displaystyle\sum_{i=1}^r\frac1{p_i^{-1}+\cdots+p_i^{-h}}}.
\end{align*}
For the variance, we compute $G_1'(1)$ as
\begin{align*}
  G_1'(1)&= \E(B_1)^2\sum_{i=1}^r \frac{(p_i - hp_i^{h})(1-p_i^h) +
      (p_i-p_i^h)h p_i^{h}}{(1-p_i^h)^2}\\
  &=\E(B_1)^2\sum_{i=1}^r
    \frac{p_i - hp_i^h - p_i^{h+1}+hp_i^{2h}
      + hp_i^{h+1}-hp_i^{2h}}{(1-p_i^h)^2}\\
  &=\E(B_1)^2\sum_{i=1}^r
    \frac{p_i(1-p_i^{h}) - hp_i^h(1-p_i)}{(1-p_i^h)^2}\\
  &=\E(B_1)^2\biggl(\sum_{i=1}^r
    \frac{p_i}{1-p_i^h}
    -h\sum_{i=1}^r\frac{ p_i^h(1-p_i)}{(1-p_i^h)^2}\biggr).
\end{align*}
By \eqref{eq:variance}, we obtain
\begin{align*}
  \V(B_1) &= 2G_1'(1) + G_1(1)- G_1(1)^2 \\
  &=
  \E(B_1)^2\biggl( -1 + 2\sum_{i=1}^r
    \frac{p_i}{1-p_i^h} - 2h\sum_{i=1}^r\frac{ p_i^h(1-p_i)}{(1-p_i^h)^2} \\
    &\qquad\qquad\qquad\qquad\qquad+
    \sum_{i=1}^r \frac{p_i^h(1-p_i)}{1-p_i^h} \biggr)\\
    &=\E(B_1)^2\biggl(\sum_{i=1}^r\frac{-p_i+p_i^{h+1}
      +2p_i+p_i^h-p_i^{h+1}}{1-p_i^h} \\
    &\qquad\qquad\qquad\qquad\qquad- 2h\sum_{i=1}^r\frac{ p_i^h(1-p_i)}{(1-p_i^h)^2}\biggr)\\
  &=\E(B_1)^2\biggl(\sum_{i=1}^r\frac{p_i +p_i^h}{1-p_i^h} - 2h\sum_{i=1}^r\frac{ p_i^h(1-p_i)}{(1-p_i^h)^2}\biggr).
\end{align*}
Together with \eqref{eq:first-h-run-expectation}, we obtain~\eqref{eq:first-h-run-variance}.
\end{proof}

\section{Expectation of the first occurrence of \texorpdfstring{$h$}{h}-runs of \texorpdfstring{$j$}{j} letters}
In this section, we consider the first position where $j$ of the letters $1$,
\ldots, $r$ had an $h$-run. In the terminology of
Section~\ref{section:preliminaries}, this corresponds to the random variable
$B_j$.

We prove the following theorem on the expectation of $B_j$.

\begin{theorem}\label{theorem:j-h-runs-expectation-general}
  For $ i\in \calA$, let
  \begin{equation}\label{eq:expectation-j-h-runs-definition-a-g}
      \alpha_i:=\frac{p_i-p_i^h}{1-p_i}, \qquad
      \gamma_i:=\frac{p_i}{1-p_i}
  \end{equation}
  and let $A_i$ and $\Gamma_i$ be the substitution operators mapping the
  variable $v_i$ to $\alpha_i$ and $\gamma_i$, respectively.

  Then the expectation of the first occurrence of $h$-runs of exactly $j$ letters is
    \begin{equation}\label{eq:expectation-j-h-runs-operator-difference}
      \E(B_j)=\biggl(\sum_{q=0}^{j-1}[y^q] \prod_{i=1}^r(y\Gamma_i+(1-y)A_i)\biggr)S(v_1,\dots,v_r),
  \end{equation}
  where $S(v_1,\ldots, v_r)$ is defined in \eqref{eq:Smirnov}.
\end{theorem}

For $j=r$, i.e., the first occurrence of $h$-runs of all letters,
\eqref{eq:expectation-j-h-runs-operator-difference} can be simplified:

\begin{corollary}\label{corollary:all-h-runs-expectation-general}
  The expectation of the first occurrence of all $h$-runs is
  \begin{equation}\label{eq:expectation-all-h-runs-operator-difference}
       \E(B_r)=\biggl(\prod_{i=1}^r\Gamma_i -\prod_{i=1}^r(\Gamma_i-A_i)\biggr)S(v_1,\dots,v_r),
  \end{equation}
  where $\Gamma_i$, $A_i$ and $S(v_1,\ldots, v_r)$ are defined in
  \eqref{eq:expectation-j-h-runs-definition-a-g} and \eqref{eq:Smirnov}, respectively.
\end{corollary}

In the case of equidistributed letters, i.e., $p_i=1/r$ for all $i$, we get the
following simple expression.

\begin{corollary}\label{corollary:all-h-runs-expectation-equidistribution}
  If $p_1=\cdots=p_r=1/r$, then the expectation of the first occurrence of all
  $h$-runs is
  \begin{equation*}
    \E(B_r)=\frac{r(r^h-1)}{r-1}H_r,
  \end{equation*}
  where $H_r$ denotes the $r$th harmonic number.
\end{corollary}

\begin{proof}[Proof of Theorem~\ref{theorem:j-h-runs-expectation-general}]
As in Section~\ref{section:preliminaries}, $Y_n$ is the number of letters
that have at least one run of length $\ge h$ within $X_1\ldots X_n$.

Arbitrary words arise from Smirnov words by replacing single letters by runs of
length at least $1$ of the same letter. In terms of generating functions, this
corresponds to substituting $v_i$ by
\begin{multline*}
  p_i z+\cdots+(p_i z)^{h-1}+u_i((p_i z)^{h}+(p_i z)^{h+1}+\cdots)\\
  =\frac{p_i z-(p_i z)^h+u_i(p_i z)^h}{1-p_i z}=\frac{p_i z+(u_i-1)(p_i z)^h}{1-p_i z}=:\beta_i(u_i,z).
\end{multline*}
As previously, $z$ counts the length of the word. The variable $u_i$ counts the
number of occurrences of (non-extensible) $m$-runs of the letter $i$ with $m\ge h$.

We now consider the probability generating function
\begin{equation*}
  F(u_1, \ldots, u_r; z)=S(\beta_1(u_1, z),\ldots, \beta_r(u_r,z)).
\end{equation*}
of all words. 

For $M\subseteq \calA$, let $E_{n, M}$ be the event that
exactly the letters in $M$ have an $h$-run in $X_1\ldots X_n$. By definition,
we have
\begin{equation}\label{eq:Y_n=q-decomposition}
  \{Y_n=q\} = \biguplus_{\substack{M\subseteq \calA\\ \abs{M}=q}}E_{n, M}
\end{equation}
for $q\in\{0, \ldots, r\}$.

We now compute $\P(E_{n, M})$ for some $M=\{i_1, \ldots, i_q\}$ of cardinality
$q$. We denote the letters not contained in $M$ by $\calA\setminus
M=\{s_1,\ldots, s_{n-q}\}$. By construction of the generating function, we have
\begin{equation}\label{eq:Y_n=r-probability-multi-extraction}
  \P(E_{n, M})=[z^n][u_{s_1}^0]\cdots [u_{s_{n-q}}^0]\sum_{m_{i_1}, \ldots, m_{i_q} \ge 0}[u_{i_1}^{m_{i_1}}]\cdots
  [u_{i_q}^{m_{i_q}}]F(u_1,\ldots, u_r; z).
\end{equation}

For any power series $H(u)$, we have
\begin{equation*}
  \sum_{m\ge 1}[u^m]H(u)=H(1)-H(0).
\end{equation*}
We therefore define the operators $\Delta_i$ and $Z_i$ by $\Delta_i
H(u_i)=H(1)-H(0)$ and $Z_i H(u_i)=H(0)$. With
these notations, \eqref{eq:Y_n=r-probability-multi-extraction} reads
\begin{equation}\label{eq:expectation-B_r-Y_n=r}
  \P(E_{n,M})=[z^n]\biggl(\prod_{i\in M}\Delta_i\prod_{i\notin M}Z_i\biggr) F(u_1,\ldots, u_r; z).
\end{equation}

Inserting this and \eqref{eq:Y_n=q-decomposition} in
\eqref{eq:expectation-general} yields
\begin{equation}\label{eq:expectation-B_j_with_z}
  \E(B_j)=\sum_{n\ge 0}[z^n] \sum_{\substack{M\subseteq \calA\\ \abs M<j}} \biggl(\prod_{i\in M}\Delta_i\prod_{i\notin M}Z_i\biggr) F(u_1,\ldots, u_r; z).
\end{equation}

Summing over all $n\ge 0$ amounts to setting $z=1$ as long as all summands are
non-singular at $z=1$. As $\abs M<j$, at least one of the $u_i$ is zero,
w.l.o.g.\ $u_1=0$. This implies that $[z^n]F(u_1,\ldots, u_r; z)\le[z^n]F(0,
1, \ldots, 1; z)< \rho^n$ for a suitable $0<\rho<1$ as the word $1^h$ is forbidden
as a factor. Thus $F(u_1, \ldots, u_r; z)$ is regular at $z=1$.

We note that $\beta_i(1, 1)=\gamma_i$ and $\beta_i(0, 1)=\alpha_i$ where
$\gamma_i$ and $\alpha_i$ are defined in
\eqref{eq:expectation-j-h-runs-definition-a-g}. Therefore, for $z=1$, the operator
$\Delta_i$ can be written as $\Gamma_i-A_i$. Similarly, $Z_i$ corresponds to
$A_i$. 

We have
\begin{equation*}
  \sum_{\substack{M\subseteq \calA\\ \abs M<j}}\prod_{i\in
    M}(\Gamma_i-A_i)\prod_{i\notin M}A_i =
  \sum_{q=0}^{j-1}[y^q]\prod_{i=1}^r(y\Gamma_i + (1-y)A_i).
\end{equation*}
Combining this with \eqref{eq:expectation-B_j_with_z} yields
\eqref{eq:expectation-j-h-runs-operator-difference}.
\end{proof}

\begin{proof}[Proof of Corollary~\ref{corollary:all-h-runs-expectation-general}]
  The polynomial $\prod_{i=1}^r(y\Gamma_i+(1-y)A_i)$ has degree $r$ in the
  variable $y$. Thus extracting all coefficients but the coefficient of $y^r$
  amounts to substituting $y=1$ and subtracting the coefficient of $y^r$, i.e.,
  \begin{equation*}
    \sum_{q=0}^{j-1}[y^q] \prod_{i=1}^r(y\Gamma_i+(1-y)A_i) =
    \prod_{i=1}^r \Gamma_i - \prod_{i=1}^r(\Gamma_i - A_i).
  \end{equation*}
  Inserting this into \eqref{eq:expectation-j-h-runs-operator-difference} yields
  \eqref{eq:expectation-all-h-runs-operator-difference}.
\end{proof}

\begin{proof}[Proof of
  Corollary~\ref{corollary:all-h-runs-expectation-equidistribution}] Setting
  $p_i=1/r$ yields
\begin{align*}
\gamma_i&=\frac{\frac1r}{1-\frac1r}=\frac1{r-1},&
\alpha_i&=\frac{\frac1r-(\frac1r)^h}{1-\frac1r}=\frac{1-\frac1{r^{h-1}}}{r-1},\\
\frac{\gamma_i}{1+\gamma_i}&=\frac 1r,&
\frac{\alpha_i}{1+\alpha_i}&=\frac{r^{h-1}-1}{r^h-1}.
\end{align*}
  Inserting this in \eqref{eq:expectation-all-h-runs-operator-difference}
  and collecting terms with $k$ occurrences of $A_i$ yields
  \begin{align*}
    \E(B_r)&=\sum_{k=1}^{r}\binom r k(-1)^{k+1}\frac{1}{1-\frac{r-k}{r}-k\frac{r^{h-1}-1}{r^h-1}}\\
    &=\frac{r(r^h-1)}{r-1}\sum_{k=1}^{r}\binom r k(-1)^{k+1}\frac{1}{k}
    =\frac{r(r^h-1)}{r-1}H_r,
  \end{align*}
  where we used the identity
  \begin{equation*}
    H_r=\sum_{k=1}^{r}\binom r k\frac{(-1)^{k+1}}{k},
  \end{equation*}
  cf.\ \cite{Larcombe-Fennessey-Koepf-French:2003:gould-no}.
\end{proof}

\begin{remark}
  Let run lengths $h_1$, \ldots, $h_r$ be given and consider occurrences of $h_i$-runs for the letter $i$. If $B_j$ is the first position $n$ such that there are exactly $j$ letters
  which had ``their'' run in $X_1\ldots X_n$, the results of
  Theorems~\ref{theorem:first-h-run-expectation} and
  \ref{theorem:j-h-runs-expectation-general} as well as
  Corollary~\ref{corollary:all-h-runs-expectation-general} remain valid when
  all $p_i^{h}$ are replaced by $p_i^{h_i}$.
\end{remark}

\section{Algorithmic Aspects}\label{sec:algorithmic}
For fixed $h$, the occurrence of an $h$-run of the variable $X_i$ can easily be detected by a
transducer automaton reading the occurrence probabilities $p_i$ and outputting
$1$ whenever the letter $i$ completes an $h$ run, see
Figure~\ref{fig:transducer-3-runs-1} for the case $r=2$, $h=3$ and $i=2$.
\begin{figure}[htbp]
  \centering
  \begin{tikzpicture}[auto, initial text=, >=latex]
    \node[state, accepting, initial, initial where=below] (v0) at (0.000000, 0.000000) {$$};
    \node[state, accepting] (v1) at (3.000000, 0.000000) {$2$};
    \node[state, accepting] (v2) at (3.000000, 3.000000) {$22$};
    \path[->] (v1) edge node[rotate=90.00, anchor=south] {$p_{2}\mid 0$} (v2);
    \path[->] (v1.185.00) edge node[rotate=360.00, anchor=north] {$p_{1}\mid 0$} (v0.355.00);
    \path[->] (v2) edge[loop right] node[rotate=90, anchor=north] {$p_{2}\mid 1$} ();
    \path[->] (v2) edge node[rotate=45.00, anchor=south] {$p_{1}\mid 0$} (v0);
    \path[->] (v0.5.00) edge node[rotate=0.00, anchor=south] {$p_{2}\mid 0$} (v1.175.00);
    \path[->] (v0) edge[loop left] node[rotate=90, anchor=south] {$p_{1}\mid 0$} ();
  \end{tikzpicture}
  \caption{Transducer detecting $3$-runs of the letter $1$.}
  \label{fig:transducer-3-runs-1}
\end{figure}
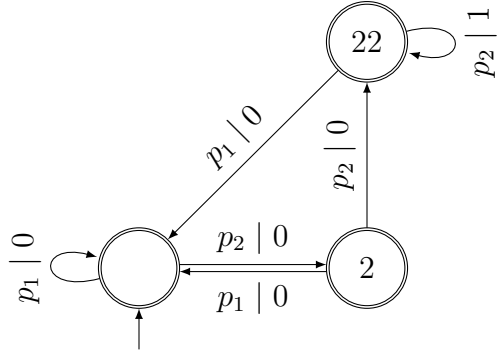

The same can be done for the first occurrence of any $h$-run, see
Figure~\ref{fig:transducer-3-runs-any} for $r=2$ and $h=3$.
\begin{figure}[htbp]
  \centering
  \begin{tikzpicture}[auto, initial text=, >=latex]
    \node[state, accepting, initial, initial where=above] (v0) at (0.000000, 1.000000) {$$};
    \node[state, accepting] (v1) at (3.000000, -0.500000) {$2$};
    \node[state, accepting] (v2) at (-3.000000, -0.500000) {$1$};
    \node[state, accepting] (v3) at (0.000000, -3.000000) {$22$};
    \node[state, accepting] (v4) at (0.000000, 3.000000) {$11$};
    \path[->] (v2.5.00) edge node[rotate=0.00, anchor=south] {$p_{2}\mid 0$} (v1.175.00);
    \path[->] (v2) edge node[rotate=49.40, anchor=south] {$p_{1}\mid 0$} (v4);
    \path[->] (v1.185.00) edge node[rotate=360.00, anchor=north] {$p_{1}\mid 0$} (v2.355.00);
    \path[->] (v1) edge node[rotate=39.81, anchor=south] {$p_{2}\mid 0$} (v3);
    \path[->] (v4) edge node[rotate=-49.40, anchor=south] {$p_{2}\mid 0$} (v1);
    \path[->] (v4) edge[loop above] node {$p_{1}\mid 1$} ();
    \path[->] (v3) edge node[rotate=320.19, anchor=south] {$p_{1}\mid 0$} (v2);
    \path[->] (v3) edge[loop below] node {$p_{2}\mid 1$} ();
    \path[->] (v0) edge node[rotate=26.57, anchor=south] {$p_{1}\mid 0$} (v2);
    \path[->] (v0) edge node[rotate=-26.57, anchor=south] {$p_{2}\mid 0$} (v1);
  \end{tikzpicture}
  \caption{Transducer detecting the first $3$-run of any letter.}
  \label{fig:transducer-3-runs-any}
\end{figure}
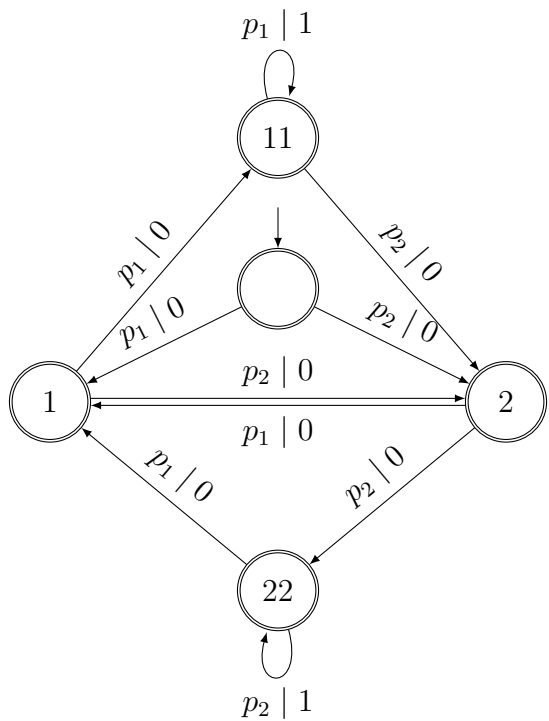

The first occurrence of $j$ runs of length $h$ could also be modelled by a
transducer.

Using the finite state machine
package~\cite{Heuberger-Krenn-Kropf:ta:finit-state} of the SageMath Mathematics
Software~\cite{Stein-others:2015:sage-mathem-6.5}, such transducers can
easily be constructed.

Accompanying this article, in \cite{Heuberger:trac-waiting-time}, an extension of SageMath to compute the expectation
and the variance of the first occurrence of a $1$ in the output of a transducer
is proposed for inclusion into SageMath.

Using this extension, the expectation and the variance of $B_1$ can be computed for
fixed $r$ and $h$ as shown in Table~\ref{tab:sage-B_1}.
\input{pygments}
\begin{table}[htbp]
\begin{center}
\begin{BVerbatim}[commandchars=\\\{\}]
\PY{k+kn}{from} \PY{n+nn}{sage.combinat} \PY{k+kn}{import} \PY{n}{finite\PYZus{}state\PYZus{}machine} \PY{k}{as} \PY{n}{FSM}

\PY{c}{# Deactivate deprecated code}
\PY{n}{FSM}\PY{o}{.}\PY{n}{FSMOldCodeTransducerCartesianProduct} \PY{o}{=} \PY{n+nb+bp}{False}
\PY{n}{FSM}\PY{o}{.}\PY{n}{FSMOldProcessOutput} \PY{o}{=} \PY{n+nb+bp}{False}

\PY{c}{# Construct the polynomial ring and set up q}
\PY{n}{R}\PY{o}{.}\PY{o}{<}\PY{n}{p}\PY{o}{>} \PY{o}{=} \PY{n}{QQ}\PY{p}{[}\PY{p}{]}
\PY{n}{q} \PY{o}{=} \PY{l+m+mi}{1} \PY{o}{-} \PY{n}{p}

\PY{c}{# Construct the Transducers detecting runs of single}
\PY{c}{# letters.  [p, p, p] is the block to detect, [p, q]}
\PY{c}{# the alphabet}
\PY{n}{p\PYZus{}runs} \PY{o}{=} \PY{n}{transducers}\PY{o}{.}\PY{n}{CountSubblockOccurrences}\PY{p}{(}
    \PY{p}{[}\PY{n}{p}\PY{p}{,} \PY{n}{p}\PY{p}{,} \PY{n}{p}\PY{p}{]}\PY{p}{,} \PY{p}{[}\PY{n}{p}\PY{p}{,} \PY{n}{q}\PY{p}{]}\PY{p}{)}
\PY{n}{q\PYZus{}runs} \PY{o}{=} \PY{n}{transducers}\PY{o}{.}\PY{n}{CountSubblockOccurrences}\PY{p}{(}
    \PY{p}{[}\PY{n}{q}\PY{p}{,} \PY{n}{q}\PY{p}{,} \PY{n}{q}\PY{p}{]}\PY{p}{,} \PY{p}{[}\PY{n}{p}\PY{p}{,} \PY{n}{q}\PY{p}{]}\PY{p}{)}

\PY{c}{# In order to detect runs of both letters, build the}
\PY{c}{# cartesian product ...}
\PY{n}{both\PYZus{}runs} \PY{o}{=} \PY{n}{p\PYZus{}runs}\PY{o}{.}\PY{n}{cartesian\PYZus{}product}\PY{p}{(}\PY{n}{q\PYZus{}runs}\PY{p}{)}
\PY{c}{# ... and add up the output by concatenating with}
\PY{c}{# the predefined "add" transducer on the alphabet}
\PY{c}{# [0, 1] We use the Python convention that any}
\PY{c}{# non-zero integer evaluates to True in boolean}
\PY{c}{# context.}
\PY{n}{first\PYZus{}run} \PY{o}{=} \PY{n}{transducers}\PY{o}{.}\PY{n}{add}\PY{p}{(}\PY{p}{[}\PY{l+m+mi}{0}\PY{p}{,} \PY{l+m+mi}{1}\PY{p}{]}\PY{p}{)}\PY{p}{(}\PY{n}{both\PYZus{}runs}\PY{p}{)}

\PY{c}{# Declare it as a Markov chain}
\PY{n}{first\PYZus{}run}\PY{o}{.}\PY{n}{on\PYZus{}duplicate\PYZus{}transition} \PY{o}{=} \PYZbs{}
    \PY{n}{FSM}\PY{o}{.}\PY{n}{duplicate\PYZus{}transition\PYZus{}add\PYZus{}input}
\PY{k}{print} \PY{n}{first\PYZus{}run}\PY{o}{.}\PY{n}{moments\PYZus{}waiting\PYZus{}time}\PY{p}{(}\PY{p}{)}
\end{BVerbatim}

\end{center}
\caption{Computation of the moments for $B_1$ with $r=2$ and $h=3$ in SageMath.}
\label{tab:sage-B_1}
\end{table}
The results coincide with those obtained in
Theorem~\ref{theorem:first-h-run-expectation}. For more examples, see the
documentation of \verb+moments_waiting_time+.

For $j>1$, we did not compute $\V(B_j)$ in general. For fixed $r$ and $h$, it
can be computed by this algorithmic approach.

Obviously, the SageMath method can be used for computing first occurrences of
everything which is recognisable by a transducer. On the other hand, explicit
results for general $r$ and $h$ such as our
Theorems~\ref{theorem:first-h-run-expectation} and
\ref{theorem:j-h-runs-expectation-general} cannot be obtained by that method.

\bibliographystyle{amsplainurl}
\bibliography{bib/cheub}
\end{document}